\documentclass{amsart}
\usepackage{amssymb,amsmath}
\usepackage{hyperref}
\usepackage{appendix}
\usepackage{paralist}
\usepackage{graphics}
\usepackage{epsfig}
\usepackage{graphicx}
\usepackage{epstopdf}
\usepackage[]{algorithm2e}
\usepackage{graphicx,amsxtra,amssymb, tikz}
\usepackage{amsmath}
\usepackage{epsfig}
\usepackage{hyperref}
\usepackage{appendix}
\usepackage{setspace}
\usepackage{MnSymbol,wasysym}
\newtheorem{thm}{Theorem}[section]

\newtheorem{lem}{Lemma}[section]

\numberwithin{equation}{section}
\newcommand{\Deltah}{\varDelta_{\ssy{\sf H}}}

\newcommand{\rset}{{\mathbb R}}
\newcommand{\nset}{{\mathbb N}}

\newcommand{\btheta}{{\sf\theta}}

\newcommand{\ee}{{\sf e}}
\newcommand{\HH}{{\sf H}}
\newcommand{\ssy}{\scriptscriptstyle}

\newcommand{\uu}{u}
\newcommand{\UUU}{U}

\newcommand{\gff}{{\mathfrak n}}

\newcommand{\half}{\frac{1}{2}}
\newcommand{\quarter}{\frac{1}{4}}
\newcommand{\dspace}{{\sf X}^{\circ}_{\ssy\sf H}}

\newcommand{\ddelta}{\delta_{\star}}

\newcommand{\PP}{{\sf I}^{\circ}_{\ssy\sf H}}
\begin{document}
\title[]
{A Relaxation/Finite Difference discretization\\
 of a 2D Semilinear Heat Equation\\
over a rectangular domain}
\author[]{Georgios E. Zouraris$^{\ddag}$}
\thanks
{$^{\ddag}$ Department of Mathematics and Applied Mathematics and
Applied Mathematics Laboratory (AML),
University of Crete, GR-700 13 Voutes Campus, Heraklion, Crete, Greece. 
(e-mail: georgios.zouraris@uoc.gr)}
\subjclass{65M12, 65M60}
\keywords {Relaxation Scheme, semilinear heat equation,
finite differences, Dirichlet boundary conditions,
optimal order error estimates}
%
%
%
%
\begin{abstract}
%
%
We consider an initial and Dirichlet boundary value
problem for a semilinear, two dimensional heat equation over a
rectangular domain. The problem is discretized in time
by a version of the Relaxation Scheme proposed by C. Besse
(C. R. Acad. Sci. Paris S{\'e}r. I, vol. 326 (1998)) for the nonlinear
Schr{\"o}dinger equation and in space by a standard second order
finite difference method.
The proposed method is unconditionally well-posed and its
convergence is established by proving an optimal second order
error estimate allowing a mild mesh condition to hold. 
\end{abstract}
\maketitle
%
%
%
%
%
%
%
%
\section{Introduction}
\subsection{Formulation of the problem}
Let $T>0$, $a_1,a_2,b_1,b_2\in{\mathbb R}$ with $a_2>a_1$ and
$b_2>b_1$, ${\mathcal D}:=[a_1,a_2]\times[b_1,b_2]$,
$Q:=[0,T]\times{\mathcal D}$ and $\uu\in C_{t,x_1,x_2}^{1,2,2}(Q,{\mathbb R})$
be the solution of the following initial and boundary value problem:
\begin{gather}
\uu_t=\varDelta\uu+g(\uu)+f
\quad\text{\rm on}\ \ Q,
\label{PSL_a}\\
\uu(t,x)=0\quad\forall\,(t,x)\in[0,T]\times\partial{\mathcal D},
\label{PSL_b}\\
\uu(0,x)=\uu_0(x)\quad\forall\,x\in{\mathcal D},
\label{PSL_c}
\end{gather}
where $g\in C({\mathbb R},{\mathbb R})$ with $g(0)=0$,
$f\in C(Q,{\mathbb R})$ and $\uu_0\in C({\mathcal D},{\mathbb R})$
with
\begin{equation}\label{PSL_d}
\uu_0\left|_{\ssy\partial{\mathcal D}}\right.=0.
\end{equation}
\subsection{Formulation of the numerical method}
Let $\nset$ be the set of all positive integers. For a given
$N\in\nset$, we define a uniform partition of the time interval $[0,T]$ with
time-step $\tau:=\tfrac{T}{N}$, nodes $t_n:=n\,\tau$ for $n=0,\dots,N$,
and intermediate nodes $t^{n+\half}=t_n+\tfrac{\tau}{2}$ for
$n=0,\dots,N-1$. Also, for given $J_1,J_2\in\nset$, we consider a uniform partition
of $[a_1,a_2]$ with mesh-width $h_1:=\tfrac{a_2-a_1}{J_1+1}$ and nodes
$x_{1,i}:=a_1+i\,h_1$ for $i=0,\dots,J_1+1$, and a uniform partition
of $[b_1,b_2]$ with mesh-width $h_2:=\tfrac{b_2-b_1}{J_2+1}$ and nodes
$x_{2,j}:=b_1+j\,h_2$ for $j=0,\dots,J_2+1$. 
Also, we set 
${\mathbb I}:=\{(i,j):\,\,i=0,\dots,J_1+1,\,\, j=0,\dots,J_2+1\}$,
${\mathbb I}^{\circ}:=\{(i,j):\,\,i=1,\dots,J_1,\,\, j=1,\dots,J_2\}$,
$\partial{\mathbb I}:={\mathbb I}\backslash{\mathbb I}^{\circ}$,
and introduce the discrete space
\begin{equation*}
{\dspace}:=\left\{\,V=(V_{i,j})_{(i,j)\in{\mathbb I}}\in\rset^{(J_1+2)\times(J_2+2)}:\,\, 
V_{i,j}=0\,\,\,\forall(i,j)\in\,\partial{\mathbb I}\right\},
\end{equation*}
and a discrete Laplacian operator $\Deltah:\dspace\rightarrow\dspace$ by
\begin{equation*}
(\Deltah V)_{i,j}:=\tfrac{V_{i-1,j}-2\,V_{i,j}+V_{i+1,j}}{h_1^2}
+\tfrac{V_{i,j-1}-2\,V_{i,j}+V_{i,j+1}}{h_2^2},
\quad\forall\,(i,j)\in{\mathbb I}^{\circ}
\quad\forall\,V\in\dspace.
\end{equation*}
In addition, we introduce an o\-pe\-rator
${\sf I}_{\ssy\HH}^{\circ}:C({\mathcal D})\rightarrow\dspace$,
which, for given $z\in C({\mathcal D})$, is defined by
$({\sf I}^{\circ}_{\ssy\HH}[z])_{i,j}:=z(x_{1,i},x_{2,j})$
for all $(i,j)\in{\mathbb I}^{\circ}$.
Finally, for any $W\in\dspace$, we define $g(W)\in\dspace$ by
$(g(W))_{i,j}:=g(W_{i,j})$ for all $(i,j)\in{\mathbb I}^{\circ}$.
%
%
%
\par
The Relaxation Finite Difference (RFD) method uses a standard finite
difference scheme for space discretization along with a variant of the
Relaxation Scheme for time stepping (cf. \cite{Besse1}).
\par\noindent\vskip0.2truecm\par\noindent
{\tt Step 1}: First, set
\begin{equation}\label{BRS_1}
\UUU^0:={\sf I}_{\ssy\sf H}^{\circ}\left[\uu_0\right]\in\dspace
\end{equation}
and find $\UUU^{\half}\in\dspace$ such that
\begin{equation}\label{BRS_2}
\tfrac{\UUU^{\half}-\UUU^0}{(\tau/2)}=\Deltah\UUU^{\half}+g(\uu^0)
+\PP\left[f(t^{\half},\cdot)\right].
\end{equation}
\par\noindent\vskip0.2truecm\par\noindent
{\tt Step 2}: Set
\begin{equation}\label{BRS_13}
\Phi^{\half}:=g(\UUU^{\half})\in\dspace
\end{equation}
and find $\UUU^1\in\dspace$ such that
\begin{equation}\label{BRS_12}
\begin{split}
\tfrac{\UUU^1-\UUU^0}{\tau}=\Deltah\left(\tfrac{\UUU^1+\UUU^0}{2}\right)
+\Phi^{\half}+\PP\left[\tfrac{f(t_{1},\cdot)+f(t_0,\cdot)}{2}\right].
\end{split}
\end{equation}
\par\noindent\vskip0.2truecm\par\noindent
{\tt Step 3}: For $n=1,\dots,N-1$, first set
\begin{equation}\label{BRS_3}
\Phi^{n+\half}:=2\,g(\UUU^n)-\Phi^{n-\half}\in\dspace
\end{equation}
and then find $\UUU^{n+1}\in\dspace$ such that
\begin{equation}\label{BRS_4}
\tfrac{\UUU^{n+1}-\UUU^{n}}{\tau}
=\Deltah\left(\tfrac{\UUU^{n+1}+\UUU^{n}}{2}\right)
+\Phi^{n+\half}+\PP\left[\tfrac{f(t_{n+1},\cdot)+f(t_n,\cdot)}{2}\right].
\end{equation}
%
%
\subsection{References and main results}
The Relaxation Scheme (RS) has been introduced by C. Besse \cite{Besse1}
as a linearly implicit, conservative, time stepping method for the approximation
of the solution to the nonlinear Schr{\"o}dinger equation. Convergence results
has been obtained in \cite{Besse2} and \cite{Besse3} for the (RS) time-discrete
approximations of the Cauchy problem for the nonlinear Schr{\"o}dinger equation,
which can not be concluded in the fully-discrete case, because on the one hand they
are valid for small final time $T$, and on the other hand are based on the derivation
of $H^{s+2}$ a priori bounds  (with $s>\tfrac{d}{2}$) for the time discrete
approximations.
Recently, in \cite{Georgios2}, the (RS) joined with a finite difference scheme
is proposed for the approximation of the solution to a semilinear heat equation
in the 1D case, and the corresponding error analysis arrived at optimal second order
error estimates. Here, we investigate the extension of the results obtained in
\cite{Georgios2} in the 2D case. We would like to stress that a fylly-discrete version
of the (RS) applied on a parabolic problem can be analyzed by using energy estimates,
which, however, are not efficient in the case of the nonlinear Schr{\"o}dinger equation
(see \cite{Georgios3}).
\par
To develop an error analysis for the (RFD) method, we introduce 
the modified scheme (see Section~\ref{The_MBRFD}) that follows
from the (RFD) method after mollifying properly the terms with nonlinear
structure (cf. \cite{Akrivis1},  \cite{KarMak}, \cite{Georgios1}).
%
%
For the approximations obtained from the modified scheme, we provide an optimal,
second order error estimate in the discrete $H^1-$norm at the nodes and in the
discrete $L^2-$norm at the intermediate nodes (see Theorem~\ref{BR_CB_Conv}). 
After applying an inverse inequality (see \eqref{SOBO}) and imposing a proper mesh condition
(see \eqref{BR_Xmesh}), the latter
convergence result implies that the discrete $L^{\infty}-$norm of the modified approximations
is uniformly bounded, and thus they are the same with those derived by the (RFD) method
and hence inherit their convergence properties o (see Theorem~\ref{DR_Final}),
i.e. that there exist constant $C>0$, independent of $\tau$, $h_1$, $h_2$, such that
\begin{equation*}
\max_{0\leq n\leq {\ssy N}}\left[\,
\big|\,\Phi^{n+\half}-{\sf I}_{\ssy\HH}^{\circ}[g(u(t^{n+\half},\cdot))]\,\big|_{0,\HH}
+\big|\,U^n-\PP[u(t_n,\cdot)]\,\big|_{1,\HH}\,\right]\leq\,C\,(\tau^2+h_1^2+h_2^2),
\end{equation*}
where $|\cdot|_{1,\HH}$ is a discrete $H^1-$norm 
and $|\cdot|_{0,\HH}$ is a discrete $L^2-$norm.
%
%
%
%
%
%
%
%
\section{Preliminaries}\label{Section2}
%
%
\subsection{Discrete relations}
%
%
%
%
%
%
%
%
%
We provide $\dspace$ with the discrete inner product
$(\cdot,\cdot)_{0,\HH}$ given by
$$(V,Z)_{0,\HH}:=h_1\,h_2\,\sum_{(i,j)\in{\mathbb I}^{\circ}}V_{i,j}\,Z_{i,j}
\quad\forall\,V,Z\in\dspace$$
and we shall denote by $\|\cdot\|_{0,\HH}$ its induced norm, i.e.
$\|v\|_{0,\HH}:=\left[(V,V)_{0,\HH}\right]^{\ssy 1/2}$ for $V\in\dspace$.
Also, we equip $\dspace$ with a discrete $L^{\infty}$-norm $|\cdot|_{\infty,\HH}$
defined by $|W|_{\infty,\HH}:=\max_{(i,j)\in{\mathbb I}^{\circ}}|W_{i,j}|$ for
$W\in\dspace$, and with a discrete $H^1$-type norm $|\cdot|_{1,\HH}$
given by
$$|V|_{1,\HH}:=\left[h_1\,h_2\,\sum_{j=1}^{\ssy J_2}\sum_{i=0}^{\ssy J_1}
\left|\tfrac{V_{i+1,j}-V_{i,j}}{h_1}\right|^2
+h_1\,h_2\,\sum_{i=1}^{\ssy J_1}\sum_{j=0}^{\ssy J_2}
\left|\tfrac{V_{i,j+1}-V_{i,j}}{h_2}\right|^2\right]^{1/2}
\quad\forall\,V\in\dspace.$$
\par
In the convergence analysis of the method, we will make use of the following,
easy to verify, relation
\begin{equation}\label{SELF}
(\Deltah V,V)_{0,\HH}=-|V|_{1,\HH}^2\quad\forall\,V\in\dspace,
\end{equation}
of the discrete Poincar{\'e}-Friedrishs inequality
\begin{equation}\label{dpoincare}
\|V\|_{0,\HH}\leq\,\tfrac{1}{2}\,\min\{a_2-a_1,b_2-b_1\}
\,|V|_{1,\HH}\quad\forall\,V\in\dspace,
\end{equation}
of the inverse inequality
\begin{lem}\label{sobol}
For $V\in\dspace$ it holds that
\begin{equation}\label{SOBO}
|V|_{\infty,\HH}\leq \tfrac{{\sf L}}{\sqrt{h_1+h_2}}\,|V|_{1,\HH}
\end{equation}
with ${\sf L}:=\left[\tfrac{(a_2-a_1)\,(b_2-b_1)}{\min\{a_2-a_1,b_2-b_1\}}\right]^{1/2}$.
\end{lem}
\begin{proof}
Let $V\in\dspace$ and $(i_0,j_0)\in{\mathbb I}^{\circ}$ such that $|V|_{\infty,\HH}=|V_{i_0,j_0}|$.
Since $V_{0,j_0}=V_{i_0,0}=0$, we conclude (see \cite{Georgios2}) that
\begin{equation}\label{Sobo1}
\begin{split}
|V_{i_0,j_0}|^2\leq&\,(a_2-a_1)\,h_1\sum_{i=0}^{\ssy J_1}\left|
\tfrac{V_{i+1,j_0}-V_{i,j_0}}{h_1}\right|^2\\
\leq&\,\tfrac{a_2-a_1}{h_2}\,h_1\,h_2\sum_{j=1}^{\ssy J_2}
\sum_{i=0}^{\ssy J_1}\left| \tfrac{V_{i+1,j}-V_{i,j}}{h_1}\right|^2\\
\end{split}
\end{equation}
and
\begin{equation}\label{Sobo2}
\begin{split}
|V_{i_0,j_0}|^2\leq&\,(b_2-b_1)\,h_2\sum_{j=0}^{\ssy J_2}\left|
\tfrac{V_{i_0,j+1}-V_{i_0,j}}{h_2}\right|^2\\
\leq&\,\tfrac{b_2-b_1}{h_1}\,h_1\,h_2\sum_{i=1}^{\ssy J_1}
\sum_{j=0}^{\ssy J_2}\left| \tfrac{V_{i,j+1}-V_{i,j}}{h_2}\right|^2.\\
\end{split}
\end{equation}
From \eqref{Sobo1} and \eqref{Sobo2}, we conclude that
\begin{equation*}
|V|_{\infty,\HH}^2\leq\,\tfrac{(b_2-b_1)\,(a_2-a_1)}{(a_2-a_1)\,h_1+(b_2-b_1)\,h_2}
\,|V|_{1,\HH}^2
\end{equation*}
which, easily, yields \eqref{SOBO}.
\end{proof}
%
%
\noindent
and of the following Lipschitz-type inequality:
%
%
\begin{lem}\label{Lemma_D3}
Let ${\sf g}\in C^2(\rset;\rset)$ with $\sup_{\ssy\rset}(|{\sf  g}'|+|{\sf g}''|)<+\infty$. 
Then, for $v^a,v^b,z^a,z^b\in\dspace$, it holds that
\begin{equation}\label{OPAP_1A}
\begin{split}
\|{\sf g}(v^a)-{\sf g}(v^b)-{\sf g}(z^a)+{\sf g}(z^b)\|_{0,\HH}
\leq&\,\sup_{\ssy\rset}|{\sf g}'|\,\,|z^a-z^b|_{\infty,\HH}\,\|v^b-z^b\|_{0,\HH}\\
&+\left(\sup_{\ssy\rset}|{\sf g}'|+\sup_{\ssy\rset}|{\sf g}''|
\,\,|z^a-z^b|_{\infty,\HH}\right)\,\|v^a-v^b-z^a+z^b\|_{0,\HH}.\\
\end{split}
\end{equation}
\end{lem}
%
%
%
%
\begin{proof}
It is similar to the proof of (2.9) in Lemma~2.3 in \cite{Georgios2}.
\end{proof}
%
%
%
\subsection{Consistency Errors}\label{Section3}
%
%
To simplify the notation, we set
$\uu^{n}:={\sf I}_{\ssy\HH}^{\circ}[\uu(t_{n},\cdot)]$
for $n=0,\dots,N$, and $\uu^{n+\half}:={\sf I}_{\ssy\HH}^{\circ}[\uu(t^{n+\half},\cdot)]$
for $n=0,\dots,N-1$.
%
%
\subsubsection{Consistency error in time}
For $n=1,\dots,N-1$, let ${\sf r}^n\in\dspace$ be defined by
\begin{equation}\label{NORAD_22}
\tfrac{g(\uu^{n+\half})+g(\uu^{n-\half})}{2}=g(\uu^n)+{\sf r}^n.
\end{equation}
Then, setting $\zeta(t,x):=g(u(t,x))$ and expanding, by using the Taylor formula
around $t=t_n$, we obtain:
\begin{equation*}\label{USNORTH_21}
{\sf r}^n=\tfrac{\tau^2}{2}\,\,{\sf I}_{\ssy\HH}^{\circ}\left[\int_{0}^{\half}
\left[\,(\tfrac{1}{2}-s)\,\zeta_{tt}(t_n+s\,\tau,\cdot)
+s\,\zeta_{tt}(t^{n-\half}+s\,\tau,\cdot)\right]\;ds\right],
\quad n=1,\dots,N-1,
\end{equation*}
which, easily, yields
\begin{equation}\label{SHELL_21}
\tau\,\max_{1\leq{n}\leq{\ssy N-1}}|{\sf r}^n|_{\infty,{\HH}}
+\max_{2\leq{n}\leq{\ssy N-1}}|{\sf r}^n-{\sf r}^{n-1}|_{\infty,\HH}
\leq\,{\sf C}_{\ssy\tt I}\,\tau^3\,\max_{\ssy Q}\left(|\zeta_{tt}|+|\zeta_{ttt}|\right).
\end{equation}
\par
Let ${\sf r}^{\frac{1}{4}}\in\dspace$ be defined by
\begin{equation}\label{NORAD_000}
\tfrac{\uu^{\half}-\uu^0}{(\tau/2)}={\sf I}_{\ssy\HH}^{\circ}\left[\uu_{xx}(t^{\half},\cdot)\right]
+g(\uu^0)+{\sf I}_{\ssy\HH}^{\circ}\left[f(t^{\half},\cdot)\right]+{\sf r}^{\frac{1}{4}}
\end{equation}
and, for $n=0,\dots,N-1$, let ${\sf r}^{n+\half}\in\dspace$ be specified by
\begin{equation}\label{NORAD_02}
\tfrac{\uu^{n+1}-\uu^n}{\tau}={\sf I}_{\ssy\HH}^{\circ}
\left[\tfrac{\uu_{xx}(t_{n+1},\cdot)+\uu_{xx}(t_n,\cdot)}{2}\right]
+g(\uu^{n+\half})+{\sf I}_{\ssy\HH}^{\circ}\left[\,\tfrac{f(t_{n+1},\cdot)+f(t_n,\cdot)}{2}
\,\right]+{\sf r}^{n+\half}.
\end{equation}
Using \eqref{PSL_a},  from \eqref{NORAD_000} and \eqref{NORAD_02},
we obtain
\begin{equation}\label{USNORTH_00B}
{\sf r}^{\frac{1}{4}}={\sf r}_{\ssy A}^{\frac{1}{4}}+{\sf r}_{\ssy B }^{\frac{1}{4}},
\quad\text{\rm and}\quad
{\sf r}^{n+\half}:={\sf r}_{\ssy A}^{n+\half}+{\sf r}_{\ssy B}^{n+\half},
\quad n=0,\dots,N-1,
\end{equation}
where
\begin{equation*}
{\sf r}^{\frac{1}{4}}_{\ssy A}=\tfrac{u^{\half}-u^0}{(\tau/2)}
-{\sf I}_{\ssy\HH}^{\circ}\left[u_t(t^{\half},\cdot)\right],
\quad
{\sf r}_{\ssy B}^{\frac{1}{4}}:=g(u^{\half})-g(u^0)
\end{equation*}
and
\begin{equation*}
\begin{split}
{\sf r}^{n+\half}_{\ssy A}:=&\,\tfrac{\uu^{n+1}-\uu^n}{\tau}
-{\sf I}_{\ssy\HH}^{\circ}\left[u_{t}(t^{n+\half},\cdot)\right]
-{\sf I}_{\ssy\HH}^{\circ}\left[\tfrac{\uu_t(t_{n+1},\cdot)
+\uu_t(t_{n},\cdot)}{2}-u_{t}(t^{n+\half},\cdot)\right],\\
{\sf r}^{n+\half}_{\ssy B}:=&\,\tfrac{g(u^{n+1})+g(u^n)}{2}
-g(u^{n+\half}).\\
\end{split}
\end{equation*}
Applying the Taylor formula we obtain
\begin{equation}\label{USNORTH_04A}
{\sf r}_{\ssy A}^{\quarter}=-2\,\tau\,{\sf I}^{\circ}_{\ssy\HH}\left[
\int_0^{\half}s\,u_{tt}(s\,\tau,\cdot)\;ds\right],\quad
{\sf r}^{\quarter}_{\ssy B}={\sf I}_{\ssy\HH}^{\circ}
\left[\int_{t_0}^{t_1}\zeta_t(s,\cdot)\;ds\right]
\end{equation}
and
\begin{equation}\label{USNORTH_02A}
\begin{split}
{\sf r}_{\ssy A}^{n+\half}=&\,\tfrac{\tau^2}{2}\,{\sf I}_{\ssy\HH}^{\circ}
\left[\int_0^{\half}\left[\,s^2\,\uu_{ttt}(t_n+s\,\tau,\cdot)
+(\tfrac{1}{2}-s)^2\,\uu_{ttt}(t^{n+\half}+s\,\tau,\cdot)\,\right]\;ds\right]\\
&\quad-\tfrac{\tau^2}{2}\,{\sf I}_{\ssy\HH}^{\circ}\left[
\int_0^{\frac{1}{2}}\left[\,s\,u_{ttt}(t_n+\tau\,s,\cdot)
+(\tfrac{1}{2}-s)\,u_{ttt}(t^{n+\half}+\tau\,s,\cdot)\,\right]\;ds\right],\\
{\sf r}_{\ssy B}^{n+\half}=&\,\tfrac{\tau^2}{2}\,{\sf I}_{\ssy\HH}^{\circ}\left[
\int_0^{\frac{1}{2}}\left[\,s\,\zeta_{tt}(t_n+\tau\,s,\cdot)
+(\tfrac{1}{2}-s)\,\zeta_{tt}(t^{n+\half}+\tau\,s,\cdot)\,\right]\;ds\right],
\quad n=0,\dots,N-1.\\
\end{split}
\end{equation}
Thus, from \eqref{USNORTH_00B}, \eqref{USNORTH_02A} and \eqref{USNORTH_04A},
we arrive at
\begin{equation}\label{SHELL_01}
\tau\,|{\sf r}^{\quarter}|_{\infty,\HH}
+\max_{0\leq{n}\leq{\ssy N-1}}|{\sf r}^{n+\half}|_{\infty,\HH}
\leq\,{\sf C}_{\ssy\tt I\!I}\,\tau^2\,\max_{\ssy Q}\left(|\zeta_t|+|\zeta_{tt}|+|u_{tt}|+|u_{ttt}|\right).
\end{equation}
%
%
%
\subsubsection{Consistency error in space}
%
%
Also, let ${\sf s}^{\quarter}\in\dspace$ be defined by
\begin{equation}\label{NORAD_011}
\tfrac{\uu^1-\uu^0}{(\tau/2)}=\Deltah\uu^1
+g(\uu^0)+\PP\left[f(t^{\half},\cdot)\right]+{\sf s}^{\quarter}
\end{equation}
and, for $n=0,\dots,N-1$, let ${\sf s}^{n+\half}\in\dspace$ be given by
\begin{equation}\label{NORAD_12}
\tfrac{\uu^{n+1}-\uu^n}{\tau}=\Deltah\left(\tfrac{\uu^{n+1}+\uu^n}{2}\right)
+g(\uu^{n+\half})+\PP\left[\tfrac{f(t_{n+1},\cdot)+f(t_n,\cdot)}{2}\right]
+{\sf s}^{n+\half}.
\end{equation}
Subtracting \eqref{NORAD_011} from \eqref{NORAD_000} and \eqref{NORAD_12} from
\eqref{NORAD_02}, we obtain
\begin{equation}\label{USNORTH_11}
\begin{split}
{\sf r}^{\frac{1}{4}}-{\sf s}^{\frac{1}{4}}=&\,\PP\left[\uu_{xx}(t^{\half},\cdot)\right]
-\Deltah\uu^{\half},\\
{\sf r}^{n+\half}-{\sf s}^{n+\half}=&\,\PP\left[\tfrac{\uu_{xx}(t_{n+1},\cdot)+\uu_{xx}(t_n,\cdot)}{2}\right]
-\Deltah\left(\tfrac{\uu^{n+1}+\uu^n}{2}\right),\quad n=0,\dots,N-1.
\end{split}
\end{equation}
For $t\in[0,T]$, we use of the Taylor formula with respect to the space variables around $x=(x_{1,i},x_{2,j})$
to get
\begin{equation*}\label{USARCTIC}
\begin{split}
\PP\left[\uu_{xx}(t,\cdot)\right]
-\Deltah\left({\sf I}_{\ssy\HH}[u(t,\cdot)]\right)_{i,j}
=&\,\tfrac{h_1^2}{6}\,\int_0^1(1-z)^3\,\partial_{x_1}^4\uu(t,x_{1,i}+h_1\,z,x_{2,j})\;dz\\
&\,+\tfrac{h_1^2}{6}\,\int_0^1z^3\,\partial_{x_1}^4\uu(t,x_{1,i-1}+h_1\,z,x_{2,j})\;dz\\
&\,+\tfrac{h_2^2}{6}\,\int_0^1(1-z)^3\,\partial_{x_2}^4\uu(t,x_{1,i},x_{2,j}+h_2\,z)\;dz\\
&\,+\tfrac{h_2^2}{6}\,\int_0^1z^3\,\partial_{x_2}^4\uu(t,x_{1,i},x_{2,j-1}+h_2\,z)\;dz
\quad\forall\,(i,j)\in{\mathbb I}^{\circ},\\
\end{split}
\end{equation*}
which, along with \eqref{USNORTH_11}, yields
\begin{equation}\label{SHELL_11}
|{\sf s}^{\quarter}-{\sf r}^{\quarter}|_{\infty,\HH}
+\max_{0\leq{n}\leq{\ssy N-1}}|{\sf s}^{n+\half}-{\sf r}^{n+\half}|_{\infty,\HH}
\leq\,{\sf C}_{\ssy\tt I\!I\!I}\,(h_1^2+h_2^2)\,
\max_{\ssy Q}\left(|\partial_{x_1}^4u|+|\partial_{x_2}^4u| \right).
\end{equation}
%
\section{Convergence Analysis}\label{Section4}
%
%
\subsection{A mollifier}
For $\delta>0$, let $\gff_{\delta}\in C^3(\rset;\rset)$ (cf. \cite{Akrivis1},
\cite{KarMak}, \cite{Georgios1}) be an odd fuction defined by
\begin{equation}\label{ni_defin}
\gff_{\delta}(x):=\left\{
\begin{aligned}
&x,\hskip2.40truecm
\mbox{if}\ \ x\in[0,\delta],\\
&p_{\delta}(x),\hskip1.75truecm
\mbox{if}\ \ x\in (\delta,2\delta],\\
&2\,\delta,\hskip2.22truecm\mbox{if}\ \ x> 2\delta,\\
\end{aligned}
\right.\quad\forall\,x\ge 0,
\end{equation}
%
%
%
where $p_{\delta}$ is the unique polynomial of ${\mathbb P}^7[\delta,2\delta]$
that satisfies the following conditions:
\begin{equation*}
p_{\delta}(\delta)=\delta,\,\,\,p_{\delta}'(\delta)=1,
\,\,\,p_{\delta}''(\delta)=p_{\delta}'''(\delta)=0,
\,\,\,p_{\delta}(2\,\delta)=2\,\delta,
\,\,\,p_{\delta}'(2\,\delta)=p_{\delta}''(2\,\delta)=p_{\delta}'''(2\,\delta)=0.
\end{equation*}
%
\subsection{The modified scheme}\label{The_MBRFD}
%
For given $\delta>0$, the modified version of the (RFD) method (cf. \cite{Akrivis1}, \cite{KarMak},
\cite{Georgios1}), derives approximations
$(V^n_{\delta})_{n=0}^{\ssy N}\subset\dspace$ of the solution
$\uu$ as follows:
\par\vskip0.2truecm\par\noindent
{\sf Step M1}: First, we set
\begin{equation}\label{Mstep1}
V_{\delta}^0:=U^0\quad\text{and}\quad V_{\delta}^{\half}:=U^{\half}.
\end{equation}
\par\vskip0.2truecm\par\noindent
{\sf Step M2}: Set
\begin{equation}\label{Mstep2}
\Phi^{\frac{1}{2}}_{\delta}:=g\big(\gff_{\delta}\big(V_{\delta}^{\half}\big)\big)
\end{equation}
and find $V^{1}_{\delta}\in\dspace$ such that
\begin{equation}\label{BR_CB6}
\tfrac{V^{1}_{\delta}-V^0_{\delta}}{\tau}=
\Deltah\left(\tfrac{V^{1}_{\delta}+V^0_{\delta}}{2}\right)
+\Phi_{\delta}^{\half}
+\PP\left[\tfrac{f(t_1,\cdot)+f(t_0,\cdot)}{2}\right].
\end{equation}
\par\vskip0.2truecm\par\noindent
{\sf Step M3}: For $n=1,\dots,N-1$, first define
$\Phi_{\delta}^{n+\half}\in\dspace$ by
\begin{equation}\label{BR_CB5}
\Phi_{\delta}^{n+\half}:=2\,g\big(\gff_{\delta}\big(V_{\delta}^n\big)\big)
-\Phi_{\delta}^{n-\half}
\end{equation}
and, then, find $V^{n+1}_{\delta}\in\dspace$ such
that
\begin{equation}\label{BR_CB6}
\tfrac{V^{n+1}_{\delta}-V^n_{\delta}}{\tau}=
\Deltah\left(\tfrac{V^{n+1}_{\delta}+V^n_{\delta}}{2}\right)
+\Phi_{\delta}^{n+\half}
+\PP\left[\tfrac{f(t_{n+1},\cdot)+f(t_n,\cdot)}{2}\right].
\end{equation}
%
%
\subsection{Convergence of the modified scheme}\label{Section_Conv_MBRFD}
%
%
%
%
\begin{thm}\label{BR_CB_Conv}
Let $u_{\ssy\max}:=\max\limits_{\ssy Q}|\uu|$ and
$\ddelta\ge u_{\ssy\max}$.
Then, there exist positive constants ${\sf C}^{\ssy{\sf BCV},1}_{\ddelta}$,
${\sf C}^{\ssy{\sf BCV},2}_{\ddelta}$ and
${\sf C}^{\ssy{\sf BCV},3}_{\ddelta}$, independent of $\tau$ and $h$, such that:
if $\tau\,{\sf C}^{\ssy{\sf BCV},1}_{\ddelta}\leq\tfrac{1}{2}$, then
\begin{equation}\label{BR_CB_cnv_2}
|\uu^{\half}-V_{\ddelta}^{\half}|_{1,\HH}
\leq\,{\sf C}_{\ddelta}^{\ssy{\sf BCV},2}\,\big(\tau^{\frac{3}{2}}
+\tau^{\half}\,h_1^2+\tau^{\half}\,h_2^2\big)
\end{equation}
and
\begin{equation}\label{BR_CB_cnv_1}
\max_{0\leq{m}\leq{\ssy N-1}}\|g(\uu^{m+\half})-\Phi_{\ddelta}^{m+\half}\|_{0,\HH}
+\max_{0\leq{m}\leq{\ssy N}}|\uu^m-V_{\ddelta}^m|_{1,\HH}
\leq\,{\sf C}_{\ddelta}^{\ssy{\sf BCV},3}\,\big(\tau^2+h_1^2+h_2^2\big)
\end{equation}
\end{thm}
%
%
%
\begin{proof}
To simplify the notation, we set
$\ee^m:=\uu^m-V_{\ddelta}^m$ for $m=0,\dots,N$, and
$\btheta^m:=g(u^{m+\half})-\Phi_{\ddelta}^{m+\half}$
for $m=0,\dots,N-1$.
In the sequel,  we will use the symbol $C$ to denote a generic constant that is
independent of $\tau$, $h$ and $\ddelta$, and may changes value from one line to the other.
Also, we will use the symbol $C_{\ddelta}$ to denote a generic constant that depends on
$\ddelta$ but is independent of $\tau$, $h$, and may changes value from one line to the other.
%
%
\par\noindent\vskip0.4truecm\par\noindent
$\boxed{{\tt Part\,\,\,1}:}$ Since $\ee^0=0$, after subtracting
\eqref{BRS_2} and \eqref{BRS_12} from \eqref{NORAD_011} and
\eqref{NORAD_12} (with $n=0$), respectively, we obtain
\begin{gather}
\ee^{\half}=\tfrac{\tau}{2}\,\Deltah\ee^{\half}+\tfrac{\tau}{2}\,{\sf s}^{\frac{1}{4}},
\label{BR_EE_a}\\
\ee^{1}=\tfrac{\tau}{2}\,\Deltah\ee^{1}+\tau\,\btheta^0+\tau\,{\sf s}^{\half}.
\label{BR_EE_b}
\end{gather}
Taking the $(\cdot,\cdot)_{0,\HH}-$inner product of
\eqref{BR_EE_a} with $\ee^{\half}$, and then using \eqref{SELF},
the Cauchy-Schwarz inequality, \eqref{SHELL_01},
\eqref{SHELL_11} and the arithmetic mean inequality,
we get
\begin{equation*}
\begin{split}
\|\ee^{\half}\|_{0,\HH}^2+\tfrac{\tau}{2}\,|\ee^{\half}|_{1,\HH}^2
=&\,\tfrac{\tau}{2}\,({\sf s}^{\frac{1}{4}},\ee^{\half})_{0,\HH}\\
\leq&\,\tfrac{\tau}{2}\,\left(\,\|{\sf s}^{\frac{1}{4}}-{\sf r}^{\frac{1}{4}}\|_{0,\HH}
+\|{\sf r}^{\frac{1}{4}}\|_{0,\HH}\,\right)\,\|\ee^{\half}\|_{0,\HH}\\
\leq&\,C\,\big(\tau^2+\tau\,h_1^2+\tau\,h_2^2\big)\,\|\ee^{\half}\|_{0,\HH}\\
\leq&\,C\,\big(\tau^2+\tau\,h_1^2+\tau\,h_2^2\big)^2+\tfrac{1}{2}\,\|\ee^{\half}\|_{0,\HH}^2,\\
\end{split}
\end{equation*}
which, obviously, yields
\begin{equation}\label{BR_reco_1}
\|\ee^{\half}\|_{0,\HH}^2+\tau\,|\ee^{\half}|_{1,\HH}^2
\leq\,C\,\big(\tau^2+\tau\,h_1^2+\tau\,h_2^2\big)^2.
\end{equation}
Since $\ddelta\ge u_{\ssy\max}$, using \eqref{ni_defin}
and \eqref{BR_reco_1}, we have
\begin{equation}\label{BR_reco_6}
\begin{split}
\|\btheta^0\|_{0,h}^2=&\,\big\|g(\gff_{\ddelta}(u^{\half}))
-g\big(\gff_{\ddelta}\big(V^{\half}_{\ddelta}\big)\big)\big\|_{0,h}^2\\
\leq&\,\sup_{\ssy\rset}|(g\circ\gff_{\ssy\ddelta})'|^2\,\|\ee^{\half}\|_{0,h}^2\\
\leq&\,C_{\ddelta}\,(\tau^2+\tau\,h_1^2+\tau\,h_2^2)^2.\\
\end{split}
\end{equation}
\par
Taking the $(\cdot,\cdot)_{0,\HH}-$inner product of
\eqref{BR_EE_b} with $\ee^{1}$, and then using \eqref{SELF},
the Cauchy-Schwarz inequality, \eqref{SHELL_11},
\eqref{SHELL_01},\eqref{BR_reco_6}
and the arithmetic mean inequality, we obtain
\begin{equation*}
\begin{split}
\|\ee^{1}\|_{0,\HH}^2+\tfrac{\tau}{2}\,|\ee^{1}|_{1,\HH}^2
=&\,\tau\,({\sf s}^{\half},\ee^1)_{0,\HH}+\tau\,(\btheta^0,\ee^1)_{0,\HH}\\
\leq&\,\tau\,\left(\,\|{\sf s}^{\half}-{\sf r}^{\half}\|_{0,\HH}
+\|{\sf r}^{\half}\|_{0,\HH}+\|\btheta\|_{0,\HH}\,\right)\,\|\ee^1\|_{0,\HH}\\
\leq&\,C\,\tau\,\big(\tau^2+h_1^2+h_2^2\big)\,\|\ee^1\|_{0,\HH}\\
\leq&\,C\,\tau^2\,\big(\tau^2+h_1^2+h_2^2\big)^2
+\tfrac{1}{2}\,\|\ee^1\|_{0,\HH}^2,\\
\end{split}
\end{equation*}
which, obviously, yields
\begin{equation}\label{REL1_covid}
\|\ee^1\|_{0,\HH}^2+\tau\,|\ee^1|_{1,\HH}^2
\leq\,C\,\tau^2\,\big(\tau^2+h_1^2+h_2^2\big)^2.
\end{equation}
Taking the $(\cdot,\cdot)_{0,\HH}-$inner product of
\eqref{BR_EE_b} with $\Deltah\ee^1$,
and then using \eqref{SELF}, we obtain
\begin{equation*}\label{BR_reco_2}
|\ee^1|_{1,\HH}^2+\tfrac{\tau}{2}\,\|\Deltah\ee^1\|_{0,\HH}^2
=-\tau\,({\sf s}^{\half},\Deltah\ee^1)_{0,\HH}-\tau\,(\btheta^0,\Deltah\ee^1)_{0,\HH}.
\end{equation*}
Then, we use the Cauchy-Schwarz inequality, \eqref{SHELL_11},
\eqref{SHELL_01}, \eqref{BR_reco_6} and the arithmetic mean inequality, to have
\begin{equation*}
\begin{split}
|\ee^1|_{1,\HH}^2+\tfrac{\tau}{2}\,|\Deltah\ee^{1}|_{0,\HH}^2
\leq&\,\tau\,\left(\,\|{\sf s}^{\half}-{\sf r}^{\half}\|_{0,\HH}
+\|{\sf r}^{\half}\|_{0,\HH}+\|\btheta\|_{0,\HH}\,\right)\,\|\Deltah\ee^1\|_{0,\HH}\\
\leq&\,\tau\,\left(\,h_1^2+h_2^2+\tau^2+\|\btheta\|_{0,\HH}\,\right)\,\|\Deltah\ee^1\|_{0,\HH}\\
\leq&\,C\,\tau\,\big(\tau^2+h_1^2+h_2^2\big)\,\|\Deltah\ee^1\|_{0,\HH}\\
\leq&\,C\,\tau\,\big(\tau^2+h_1^2+h_2^2\big)^2
+\tfrac{\tau}{4}\,\|\Deltah\ee^1\|_{1,\HH}^2,\\
\end{split}
\end{equation*}
which, obviously, yields
\begin{equation}\label{REL2_covid}
|\ee^1|_{1,\HH}^2+\tau\,\|\Deltah\ee^1\|_{0,\HH}^2
\leq\,C\,\tau\,\big(\tau^2+h_1^2+h_2^2\big)^2.
\end{equation}
%
%
%
%
\par\noindent\vskip0.4truecm\par\noindent
$\boxed{{\tt Part\,\,\,2}}$:
Subtracting \eqref{BR_CB5} from \eqref{NORAD_22}, and then
using \eqref{ni_defin} and the assumption $\ddelta\ge u_{\ssy\max}$,
we obtain:
\begin{equation}\label{covid_2.1}
\btheta^{n}+\btheta^{n-1}
=2\,\left[g(\gff_{\ddelta}(u^n))-g(\gff_{\ddelta}(V_{\ddelta}^n))\right]
+2\,{\sf r}^{n},\quad n=1,\dots,N-1,
\end{equation}
which, easily, yields that
\begin{equation}\label{covid_2.2}
\btheta^{n}-\btheta^{n-2}
=2\,{\sf R}^n+2\,({\sf r}^n-{\sf r}^{n-1}),\quad n=2,\dots,N-1,
\end{equation}
where
\begin{equation}\label{covid_2.3}
{\sf R}^n:=g(\gff_{\ddelta}(V_{\ddelta}^{n-1}))-g(\gff_{\ddelta}(V_{\ddelta}^n))
-g(\gff_{\ddelta}(u^{n-1}))+g(\gff_{\ddelta}(u^n))\in\dspace.
\end{equation}
Then, we use \eqref{OPAP_1A} (with ${\sf g}=g\circ\gff_{\ddelta}$),
\eqref{ni_defin} and the mean value theorem, to get
\begin{equation}\label{covid_2.4}
\begin{split}
\|{\sf R}^n\|_{0,\HH}\leq&\,\sup_{\ssy\rset}|(g\circ\gff_{\ssy\ddelta})'|
\,\|\ee^n-\ee^{n-1}\|_{0,\HH}\\
&\quad+\sup_{\ssy\rset}|(g\circ\gff_{\ssy\ddelta})''|\,|\uu^{n-1}-\uu^{n}|_{\infty,\HH}
\,\left(\|\ee^n-\ee^{n-1}\|_{0,\HH}+\|\ee^{n}\|_{0,\HH}\right)\\
\leq&\,C_{\ddelta}\,\left(\|\ee^n-\ee^{n-1}\|_{0,\HH}
+\tau\,\|\ee^{n}\|_{0,\HH}\right),
\quad n=2,\dots,N-1.
\end{split}
\end{equation}
Taking the $(\cdot,\cdot)_{0,\HH}-$inner product of both sides of \eqref{covid_2.2}
with $\tau\big(\btheta^{n}+\btheta^{n-2}\big)$, and then using the
Cauchy-Schwarz inequality, \eqref{covid_2.4} and \eqref{SHELL_21}, it follows that
\begin{equation*}
\begin{split}
\tau\,\|\btheta^{n}\|_{0,\HH}^2-\tau\,\|\btheta^{n-2}\|^2_{0,\HH}\leq&\,
\left(2\,\tau\,\|{\sf R}^n\|_{0,\HH}
+2\,\tau\,\|{\sf r}^n-{\sf r}^{n-1}\|_{0,\HH}\right)
\,\|\btheta^{n}+\btheta^{n-2}\|_{0,\HH}\\
\leq&\,C_{\ddelta}\,\left(\,\tau\,\|\ee^n-\ee^{n-1}\|_{0,\HH}
+\tau^2\,\|\ee^{n}\|_{0,\HH}+\tau^4\right)\,\|\btheta^{n}+\btheta^{n-2}\|_{0,\HH}\\
\leq&\,C_{\ddelta}\,\tau\,\|\ee^n-\ee^{n-1}\|_{0,\HH}
\,\left(\|\btheta^n\|_{0,\HH}+\|\btheta^{n-2}\|_{0,\HH}\right)\\
&\quad+C_{\ddelta}\,\tau^2\,\|\ee^{n}\|_{0,\HH}
\,\left(\|\btheta^{n}\|_{0,\HH}+\|\btheta^{n-2}\|_{0,\HH}\right)\\
&\quad+C\,\tau^4\,\left(\|\btheta^{n}\|_{0,\HH}+\|\btheta^{n-2}\|_{0,\HH}\right),
\quad n=2,\dots,N-1,\\
\end{split}
\end{equation*}
which, along with the application of the arithmetic mean inequality, yields
\begin{equation}\label{covid_2.5}
\begin{split}
\|\btheta^{n}\|_{0,\HH}^2+\|\btheta^{n-1}\|_{0,\HH}^2
\leq&\,\|\btheta^{n-1}\|_{0,\HH}^2+\|\btheta^{n-2}\|^2_{0,\HH}\\
&\quad+\tau^{-1}\,\|\ee^n-\ee^{n-1}\|_{0,\HH}^2+C\,\tau^5\\
&\quad+C_{\ddelta}\,\tau\,\left(\|\ee^{n}\|_{0,\HH}^2
+\|\btheta^{n}\|_{0,\HH}^2+\|\btheta^{n-2}\|_{0,\HH}^2\right),
\quad n=2,\dots,N-1.\\
\end{split}
\end{equation}
%
%
%
\par\noindent\vskip0.4truecm\par\noindent
$\boxed{{\tt Part\,\,\,3}}$:
We subtract \eqref{BR_CB6} from \eqref{NORAD_12},
to obtain
\begin{equation}\label{BR_CB_EQ}
2\,(\ee^{n+1}-\ee^n)=\tau\,\Deltah\left(\ee^{n+1}+\ee^n\right)
+2\,\tau\,{\sf s}^{n+\frac{1}{2}}+2\,\tau\,\btheta^n,
\quad n=1,\dots,N-1.
\end{equation}
Now, we take the $(\cdot,\cdot)_{0,{\sf H}}-$inner product
of \eqref{BR_CB_EQ} with $(\ee^{n+1}-\ee^{n})$,
and then, use \eqref{SELF}, to have
\begin{equation}\label{BR_CB_Gat2}
\begin{split}
2\,\|\ee^{n+1}-\ee^{n}\|_{0,{\sf H}}^2
+\tau\,|\ee^{n+1}|_{1,{\sf H}}^2
-\tau\,|\ee^n|_{1,{\sf H}}^2
=&\,2\,\tau\,({\sf s}^{n+\frac{1}{2}},\ee^{n+1}-\ee^n)_{0,{\sf H}}\\
&\quad+2\,\tau\,(\btheta^n,\ee^{n+1}-\ee^n)_{0,{\sf H}},
\quad n=1,\dots,N-1.
\end{split}
\end{equation}
Using the Cauchy-Schwarz ine\-quality,
the arithmetic mean ine\-quality, \eqref{SHELL_01}
and \eqref{SHELL_11}, we have
\begin{equation}\label{BR_CB_Gat3}
\begin{split}
2\,\tau\,({\sf s}^{n+\frac{1}{2}},\ee^{n+1}-\ee^n)_{0,{\sf H}}
\leq&\,2\,\tau\,\left(\,\|{\sf s}^{n+\half}-{\sf r}^{n+\half}\|_{0,{\sf H}}
+\|{\sf r}^{n+\half}\|_{0,{\sf H}}\,\right)\,
\|\ee^{n+1}-\ee^{n}\|_{0,{\sf H}}\\
\leq&\,2\,\tau\,\big(\tau^2+h_1^2+h_2^2\big)\,
\|\ee^{n+1}-\ee^{n}\|_{0,{\sf H}}\\
\leq&\,C\,\tau^2\,\big(\tau^2+h_1^2+h_2^2\big)^2\,
+\tfrac{1}{6}\,\|\ee^{n+1}-\ee^{n}\|_{0,{\sf H}}^2,
\quad n=1,\dots,N-1.
\end{split}
\end{equation}
Next, we use the Cauchy-Schwarz inequality,  \eqref{dpoincare},
\eqref{ni_defin} and the arithmetic mean inequality, to get
\begin{equation}\label{BR_CB_Gat4}
\begin{split}
2\,\tau\,(\btheta^n,\ee^{n+1}-\ee^n)_{0,\HH}\leq&\,2\,\tau
\,\|\btheta^n\|_{0,\HH}\,\|\ee^{n+1}-\ee^{n}\|_{0,\HH}\\
\leq&\,6\,\tau^2\,\|\btheta^{n}\|_{0,\HH}^2
+\tfrac{1}{6}\,\|\ee^{n+1}-\ee^{n}\|_{0,\HH}^2,
\quad n=1,\dots,N-1.\\
\end{split}
\end{equation}
From \eqref{BR_CB_Gat2}, \eqref{BR_CB_Gat3} and
\eqref{BR_CB_Gat4}, we conclude that
\begin{equation}\label{BR_CB_Gat5.1}
\begin{split}
\tau^{-1}\,\|\ee^{n+1}-\ee^{n}\|_{0,\HH}^2
+|\ee^{n+1}|_{1,{\sf H}}^2\leq&\,|\ee^n|_{1,\HH}^2
+6\,\tau\,\|\btheta^{n}\|_{0,\HH}^2
+C\,\tau\,\big(\tau^2+h_1^2+h_2^2\big)^2,
\quad n=1,\dots,N-1.\\
\end{split}
\end{equation}
%
%
%
%
%
\par\noindent\vskip0.4truecm\par\noindent
$\boxed{{\tt Part\,\,\,4}}$:
Let us, first, introduce the error quantities:
\begin{equation}\label{BR_Zerror}
{\sf E}^m:=\tau^{-1}\,\|\ee^{m}-\ee^{m-1}\|_{0,\HH}^2
+|\ee^m|_{1,\HH}^2
+\|\btheta^{m-1}\|_{0,\HH}^2
+\|\btheta^{m-2}\|_{0,\HH}^2,\quad m=2,\dots,N.
\end{equation}
Then, from \eqref{BR_CB_Gat5.1}, \eqref{covid_2.5} and \eqref{dpoincare}, we
conclude that there exists a constant ${\sf C}_{\ddelta}^{\ssy{\sf BR, I}}>0$
such that:
\begin{equation}\label{BR_CB_Gat9}
(1-{\sf C}_{\ddelta}^{\ssy{\sf BR, I}}\,\tau)\,{\sf E}^{n+1}
\leq\,(1+{\sf C}_{\ddelta}^{\ssy{\sf BR, I}}\,\tau)\,{\sf E}^n
+C\,\tau\,\big(\tau^2+h_1^2+h_2^2\big)^2,\quad n=2,\dots,N-1.
\end{equation}
Assuming that $\tau\,{\sf C}^{\ssy{\sf BR, I}}_{\ddelta}\leq\frac{1}{2}$,
a standard discrete Gronwall argument based on \eqref{BR_CB_Gat9} yields
\begin{equation}\label{BR_CB_Gat10}
\begin{split}
\max_{2\leq{m}\leq{\ssy N}}{\sf E}^m
\leq&\,C_{\ddelta}\,\left[\,{\sf E}^2+\big(\tau^2+h_1^2+h_2^2\big)^2\,\right]\\
\leq&\,C_{\ddelta}\,\left[
\tau^{-1}\,\|\ee^{2}-\ee^{1}\|_{0,\HH}^2
+|\ee^2|_{1,\HH}^2
+\|\btheta^{1}\|_{0,\HH}^2
+\|\btheta^{0}\|_{0,\HH}^2
+\big(\tau^2+h_1^2+h_2^2\big)^2\,\right].\\
\end{split}
\end{equation}
Setting $n=1$ in \eqref{covid_2.1} and then using \eqref{SHELL_21},
\eqref{BR_reco_6}, \eqref{REL1_covid}, we get
\begin{equation}\label{BR_Papagalini}
\begin{split}
\|\btheta^1\|_{0,\HH}\leq&\,\|\btheta^0\|_{0,\HH}
+2\,\sup_{\ssy\rset}|(g\circ\gff_{\ddelta})'|\,\|\ee^1\|_{0,\HH}
+2\,\|{\sf r}^1\|_{0,\HH}\\
\leq&\,C_{\ddelta}\,\big(\tau^2+\tau\,h_1^2+\tau\,h_2^2\big).\\
\end{split}
\end{equation}
%
%
%
Also, setting $n=1$ in \eqref{BR_CB_Gat5.1}, and then
using \eqref{BR_Papagalini} and \eqref{REL2_covid},
we have
\begin{equation}\label{BR_reco_8}
\begin{split}
\tau^{-1}\,\|\ee^2-\ee^1\|_{0,\HH}^2+|\ee^{2}|_{1,\HH}^2
\leq&\,C\,\tau\,\big(\tau^2+h_1^2+h_2^2\big)^2
+|\ee^{1}|_{1,\HH}^2+6\,\tau\,\|\btheta^1\|_{0,\HH}^2\\
\leq&\,C_{\ddelta}\,\tau\,\big(\tau^2+h_1^2+h_2^2\big)^2.\\
\end{split}
\end{equation}
Thus, \eqref{BR_CB_Gat10}, \eqref{BR_reco_8},
\eqref{BR_Papagalini} and \eqref{BR_reco_6}
yield
\begin{equation}\label{BR_main_res_2}
\max_{2\leq{m}\leq{\ssy N}}{\sf E}^m\leq\,C_{\ddelta}\,\big(\tau^2+h_1^2+h_2^2\big)^2.
\end{equation}
\par
Thus, \eqref{BR_CB_cnv_2} follows from \eqref{BR_reco_1}. Since $\ee^0=0$, 
\eqref{BR_CB_cnv_1} is established, easily,
from \eqref{BR_main_res_2}, \eqref{BR_Zerror}  and \eqref{REL2_covid}.
%
%
%
%
%
%
%
%
\end{proof}
%
%
%
\subsection{Convergence of the (RFD) method}
%
%
%
%
\begin{thm}\label{DR_Final}
Let $u_{\ssy\max}:=\max\limits_{\ssy Q}|\uu|$,
$\ddelta\ge\,2\,u_{\ssy\max}$,
${\sf L}$ be the constant in the inequality in Lemma~\ref{sobol},
and ${\sf C}^{\ssy{\sf BCV},1}_{\ddelta}$
${\sf C}^{\ssy{\sf BCV},2}_{\ddelta}$,
${\sf C}^{\ssy{\sf BCV},3}_{\ddelta}$ be the positive constants
specified in Theorem~\ref{BR_CB_Conv}. If
\begin{equation}\label{BR_Xmesh}
\tau\,{\sf C}^{\ssy{\sf BCV},1}_{\ddelta}\leq\tfrac{1}{2},
\quad
{\sf C}_{\ddelta}^{\ssy{\sf BCV},2}\,{\sf L}\,\,\,\left[
\tfrac{\tau^2+h_1^2+h_2^2}{\sqrt{h_1+h_2}}\right]
\leq\,\tfrac{\ddelta}{2},
\quad
{\sf C}_{\ddelta}^{\ssy{\sf BCV},3}\,{\sf L}\,\,\,\left[
\tfrac{\tau^{\frac{3}{2}}+\tau^{\half}\,h_1^2+\tau^{\half}\,h_2^2}{\sqrt{h_1+h_2}}\right]
\leq\,\tfrac{\ddelta}{2},
\end{equation}
%
%
%
%
then
\begin{equation}\label{BR_true_1}
\max_{0\leq{m}\leq{\ssy N-1}}|g(\uu^{m+\half})-\Phi^{m+\half}|_{0,\HH}
+\max_{0\leq{m}\leq{\ssy N}}|\uu^m-\UUU^m|_{1,\HH}
\leq\,{\sf C}_{\ddelta}^{\ssy{\sf BCV},3}\,(\tau^2+h_1^2+h_2^2).
\end{equation}
\end{thm}
%
%
%
%
\begin{proof}
Since $\ddelta\ge 2\,u_{\ssy\max}$, the convergence estimates
\eqref{BR_CB_cnv_2} and \eqref{BR_CB_cnv_1},
the discrete Sobolev inequality \eqref{SOBO}
and the mesh size conditions \eqref{BR_Xmesh} imply that
\begin{equation*}
\begin{split}
\big|V_{\ddelta}^{\half}\big|_{\infty,\HH}\leq&\,\big|\uu^{\half}
-V^{\half}_{\ddelta}\big|_{\infty,\HH}+\big|\uu^{\half}|_{\infty,\HH}\\
\leq&\,{\sf L}\,(h_1+h_2)^{-\half}\,\big|\uu^{\half}
-V^{\half}_{\ddelta}\big|_{1,\HH}+u_{\ssy\max}\\
\leq&\,{\sf C}_{\ddelta}^{\ssy{\sf BCV},2}\,{\sf L}\,(h_1+h_2)^{-\half}
\,(\tau^{\frac{3}{2}}+\tau^{\half}\,h_1^{2}+\tau^{\half}\,h_2^2)+\tfrac{\ddelta}{2}\\
\leq&\,\ddelta,
\end{split}
\end{equation*}
and
\begin{equation*}
\begin{split}
\big|V_{\ddelta}^{n}\big|_{\infty,\HH}\leq&\,\big|\uu^{n}
-V^{n}_{\ddelta}\big|_{\infty,\HH}+\big|\uu^{n}|_{\infty,\HH}\\
\leq&\,{\sf L}\,(h_1+h_2)^{-\half}\,\big|\uu^{n}
-V^{n}_{\ddelta}\big|_{1,\HH}+u_{\ssy\max}\\
\leq&\,{\sf C}_{\ddelta}^{\ssy{\sf BCV},3}\,{\sf L}
\,(\tau^2+h_1^2+h_2^2)\,(h_1+h_2)^{-\half}+\tfrac{\ddelta}{2}\\
\leq&\,\ddelta,\quad n=1,\dots,N,
\end{split}
\end{equation*}
which, along with \eqref{ni_defin} and \eqref{Mstep1}, yield
\begin{equation}\label{BR_crucial_1}
\gff_{\ddelta}(V^{\half}_{\ddelta})=V^{\half}_{\ddelta},\quad
\gff_{\ddelta}\big(V_{\ddelta}^{n}\big)
=V_{\ddelta}^{n},\quad n=1,\dots,N.
\end{equation}
Thus, for $\delta=\ddelta$, the modified (RFD) approximations defined
in Section~\ref{The_MBRFD} are the (RFD) approximations defined by
\eqref{BRS_1}-\eqref{BRS_4},
and the error estimate \eqref{BR_true_1} follow as a natural outcome 
of \eqref{BR_CB_cnv_1}.
\end{proof}
%
%
%


\begin{thebibliography}{99}
%
%
%
\bibitem{Akrivis1} G. D. Akrivis,
{\em Finite difference discretization of the cubic Schr{\"o}dinger equation},
IMA J. Numer. Anal. {\bf 13} (1993), 115-124.
%
%
%
%
%
%
%
%
%
%
%
%
\bibitem{Besse1} C. Besse,
{\em Sch{\'e}ma de relaxation pour l' {\'e}quation de Schr{\"o}dinger non lin{\'e}aire et les syst{\`e}mes
de Davey et Stewartson}, C. R. Acad. Sci. Paris S{\'e}r. I {\bf 326} (1998), 1427-1432.
%
%
%
\bibitem{Besse2} C. Besse,
{\em A relaxation scheme for the nonlinear Schr{\"o}dinger equation},
SIAM J. Numer. Anal. {\bf 42} (2004), 934-952.
%
%
%
\bibitem{Besse3}
C. Besse, S. Descombes, G. Dujardin and I. Lacroix-Violet, {\em
Energy preserving methods for nonlinear Schr{\"o}dinger equations},
arXiv:1812.04890.
%
%
%
%
%
%
%
%
%
%
%
%
\bibitem{KarMak} O. Karakashian and Ch. Makridakis,
{\em A space-time finite element method for the nonlinear Schr{\"o}dinger equation: The discontinuous
Galerkin method}, Math. Comp. {\bf 67} (1998), 479-499.
%
%
%
%
%
%
%
%
%
%
%
%
%
%
\bibitem{Georgios1}
G. E. Zouraris,
{\em On the convergence of a linear two-step finite element
method for the nonlinear Schr{\"o}dinger equation},
Math. Model. Numer. Anal. \textbf{35} (2001), 389-405.
%
%
%
%
\bibitem{Georgios2}
G. E. Zouraris,
{\em Error Estimation of the Besse Relaxation Scheme
for a Semilinear Heat Equation},
arXiv:1812.09273.
%
%
%
%
%
\bibitem{Georgios3}
G. E. Zouraris,
{\em Error Estimation of the Relaxation Finite Difference Scheme for the
Nonlinear Schrödinger Equation}, arXiv:2002.09605.
%
%
%
%
\end{thebibliography}
\end{document}